\documentclass{article}
                                             
\usepackage{cite}      
\usepackage{graphicx}
\usepackage{float}
\usepackage{subfigure} 
\usepackage{amsmath}
\usepackage{amsthm}
\usepackage{amssymb}
\usepackage{enumerate}
\usepackage{deleq}   
\usepackage{type1cm} 
\usepackage[usenames,dvipsnames]{color}
\usepackage{pictexwd,dcpic}

\RequirePackage[left=1in,right=1in,top=1in,bottom=1in]{geometry}

\newtheorem{defnt}{Definition}
\newtheorem{remark}{Remark}
\newtheorem{thm}{Theorem}
\newtheorem{lemma}{Lemma}
\newtheorem{example}{Example}
\newtheorem{cor}{Corollary}
\newtheorem{prop}{Proposition}
\newtheorem{prob}{Problem}

\def\({\left(}
\def\){\right)}

\DeclareMathOperator*{\argmin}{arg\,min}

\title{Algebraic Decompositions of {DP} Problems with Linear Dynamics}
\author{Manolis C. Tsakiris and Danielle C.~Tarraf\footnote{The authors are 
with the Department of Electrical \& Computer Engineering Department at 
The Johns Hopkins University, Baltimore, MD, 21218  
(mtsakir1@jhu.edu, dtarraf@jhu.edu).}}
\date{}

\begin{document}

\maketitle
                    
\begin{abstract}

Inspired by rational canonical forms,
we introduce and analyze two decompositions of dynamic programming (DP) problems for systems with linear dynamics.
Specifically, we consider both finite and infinite horizon DP problems in which the dynamics 
are linear, the cost function depends only on the state,
and the state-space is finite dimensional but defined over an arbitrary algebraic field.
Starting from the natural decomposition of the state-space
into the direct sum of subspaces that are invariant under the system's linear transformation, 
and assuming that the cost functions exhibit an additive structure compatible with this decomposition,
we extract from the original DP problem two distinct families of smaller DP problems, 
each associated with a system evolving on an invariant subspace of the original state-space. 
We propose that each of these families constitutes a decomposition of the original problem
when the optimal policy and value function of the original problem can be reconstructed from 
the optimal policies and value functions of the individual subproblems in the family.  
We derive necessary and sufficient conditions for these decompositions to exist both in the finite and infinite horizon cases. 
We also propose a readily verifiable sufficient condition under which the first decomposition exists,
and we show that the first notion of decomposition is generally stronger than the second.
\end{abstract}

\color{black}
\section{Introduction}
\label{Sec:Introduction}

Dynamic programming (DP), pioneered by Bellman \cite{BOOK:Bellma1957}, 
has found wide-ranging applications in diverse areas.
The Principle of Optimality results in a general solution approach that is intuitive.
The case where the state-space is Euclidean, 
the underlying dynamics are linear,
and the cost function is quadratic has been particularly well studied:
It admits an elegant closed form solution obtained 
by solving an appropriate algebraic Ricatti equation in the infinite horizon case,
and admits a time-dependent closed form solution obtained by recursively solving 
the discrete-time Ricatti equation in the finite horizon case \cite{BOOK:Bertse2005}. 
Nonetheless, DP suffers from the curse of dimensionality:
Indeed, the computational complexity of the DP algorithm increases exponentially 
with the dimensions of the underlying state and input spaces.

A natural way of alleviating this problem is by decomposing the problem into smaller subproblems, 
whose solutions are subsequently combined to yield an exact
or suboptimal solution for the original problem \cite{BOOK:Bertse2005}, \cite{BOOK:Powell2011}. 
Various notions of DP decomposition have been studied,
often inspired by the context of the problem such as
operations research \cite{JOUR:Topaloglu2010}, \cite{JOUR:PotWas1987}
and circuit design \cite{JOUR:AlpKah1995}.
Similar decompositions also arise in decentralized and distributed control problems
\cite{JOUR:Rotkowitz2006}, \cite{CONF:SwiLal2010}, \cite{JOUR:ShaPar2013}, \cite{JOUR:LanDel2010}, 
\cite{JOUR:Jadbabaie2008},  \cite{JOUR:BamPagDah2002}, \cite{JOUR:RecDAnd2004},  \cite{CONF:Rantze2009}.
Group theoretic \cite{JOUR:BouDeaGol2000},
graph theoretic  \cite{JOUR:LanChaDAnd2004},
combinatorial approaches \cite{BOOKCHAPTER:RooBodRoss2009},  \cite{JOUR:BouDeaGol2000} 
and bisimulation-based model reduction in the context of Markov Decision Processes \cite{JOUR:GivDeaGre2003}
have also been considered. 
Approximate methods have also been used to decompose DP problems, 
for example by suitably approximating constraints to achieve separability \cite{JOUR:Bertsekas2007}. 

Inspired by rational canonical forms,
we introduce and study new\footnote{Preliminary versions of the results for the finite horizon case 
appeared in \cite{CONF:TsaTar2012a, CONF:TsaTar2012b}.} 
decompositions of DP problems for systems with linear dynamics.
Specifically, we consider both finite and infinite horizon DP problems in which the dynamics 
are linear and the cost function depends only on the state.
The state-space is finite dimensional but defined over an arbitrary algebraic field.
Starting from the natural decomposition of the state-space
into the direct sum of subspaces that are invariant under the system's linear transformation, 
and assuming that the cost functions exhibit an additive structure compatible with this decomposition,
we propose two notions of decomposition. 
In particular, we extract from the original DP problem two distinct families of smaller DP problems, 
each associated with a system evolving on a distinct invariant subspace. 
We propose that each of these families constitutes a decomposition of the original problem
when the optimal policy and value function of the original problem can be reconstructed from 
the optimal policies and value functions of the individual problems in the family.  
We derive necessary and sufficient conditions for these decompositions to exist both in the finite and infinite horizon cases. 
We also propose a readily verifiable sufficient condition under which the first decomposition exists,
and we show that the first notion of decomposition is generally stronger than the second, 
thereby suggesting that further research should focus on this family of problems.

\color{black}
We emphasize that our notions of decomposition involve no approximation,
and are motivated by the desire to utilize algebraic structures inherent in the dynamics to reduce complexity.
The finite state setting remains our primary interest, 
in synergy with our past and ongoing work on analysis and synthesis of finite state machines 
\cite{CONF:TaDaMe2005, JOUR:TaMeDa2008, JOUR:TaMeDa2011, BOOKCHAPTER:TaMeDa2007},
and their use as simple, approximate models of more complex systems over finite alphabets 
\cite{JOUR:Tarraf2012, JPRE:Tarraf2013, CONF:TarDuf2011}.
Nonetheless, since our results are applicable to the general setting 
of finite dimensional but otherwise arbitrary state-spaces,
we present them as such,
while highlighting the complexity reduction achieved in the finite state setting of interest.

The manuscript is organized as follows:
We begin in Section \ref{Sec:Setup} by describing the problem setup and assumptions,
and we state the problem of interest.
We present two families of smaller DP problems in Section \ref{Sec:Decompositions},
and we propose relevant new notions of DP decomposition.
We state our main results in Section \ref{Sec:MainResults},
present a full derivation in Section \ref{Sec:Derivation} and a set of illustrative examples in Section \ref{Sec:Examples}, 
and conclude with directions for future work in Section \ref{Sec:Conclusions}.

{\it Notation:}
$\mathbb{Z}$, $\mathbb{Z}_+$, $\mathbb{R}$ and $\mathbb{R}_+$ denote the set of integers, nonnegative integers, reals 
and nonnegative reals, respectively. 
For $\alpha \in \mathbb{R}$, $\lceil \alpha \rceil$ denotes the \emph{ceiling} of $\alpha$, 
that is, the smallest integer that is greater than or equal to $\alpha$.
For sets $\mathcal{X}$ and $\mathcal{Y}$, 
$\mathcal{X}^{\mathcal{Y}}$ denotes the set of all maps from $\mathcal{Y}$ to $\mathcal{X}$.
In particular, for set $\mathcal{X}$ and index set $\mathcal{T}$,
$\mathcal{X}^{\mathcal{T}}$ denotes the set of all sequences over $\mathcal{X}$ indexed by $\mathcal{T}$,
and $\{x_t\}_{t \in I}$ or $x$ (with some abuse of notation) interchangeably denote an element of $\mathcal{X}^{\mathcal{T}}$.
For $f: \mathcal{X} \rightarrow \mathbb{R}_+$,
$\displaystyle \argmin_{x \in \mathcal{X}} f(x)$ denotes the set of arguments $x$ in $\mathcal{X}$ that minimize $f$.
For maps $f: \mathcal{X} \rightarrow \mathcal{Y}$ and $g: \mathcal{Y} \rightarrow \mathcal{W}$,
$g \circ f$ denotes the composite map from $\mathcal{X}$ to $\mathcal{W}$ defined by $g \circ f (x) = g(f(x))$. 
For vectors $x_1,\hdots,x_k$ of vector space $\mathcal{X}$ over field $\mathcal{F}$, 
$<x_1,\hdots,x_k>_{\mathcal{F}}$ denotes the subspace of $\mathcal{X}$ 
spanned by $x_1,\hdots,x_k$ over $\mathcal{F}$. 
$x'$ denotes the transpose of coordinate vector $x$.
For vector space $\mathcal{X}$, index set $\mathcal{T}$, 
and $x$ and $y$ in $\mathcal{X}^{\mathcal{T}}$,
$x+ y$ denotes an element $z \in \mathcal{X}^{\mathcal{T}}$ such that
$z_t = x_t + y_t$, $\forall t \in \mathcal{T}$.
For linear operator $B$, $\mathcal{R}(B)$ and $\mathcal{N}(B)$ denote the
\emph{range space} and \emph{null space}, respectively. 
$\oplus$ denotes the \emph{direct sum} of subspaces. 
Let $\sigma:\mathcal{X} \rightarrow A$ be a map of vector space $\mathcal{X}$ into some nonempty set $A$,
and assume that $\mathcal{X}=\mathcal{S} \oplus \mathcal{V}$.
$\sigma|_{\mathcal{S}}$ denotes the restriction of the action of $\sigma$ on $\mathcal{S}$, 
defined by  $\sigma|_{\mathcal{S}}(x)=\sigma(s)$ where 
$x=s+v, \, s \in \mathcal{S}, \, v \in \mathcal{V}$ is the unique decomposition of $x$.
For $\displaystyle \mathcal{X} = \oplus_{i \in I} \mathcal{X}_i$,
$\rho_i$ denotes the $i^{th}$ `projection' map,
$\rho_i : \mathcal{X} \rightarrow \mathcal{X}_i$, 
defined by $\rho_{i}(x) = x_i$ where $x_i$ is the unique component of $x$ in $\mathcal{X}_i$,
while 
$[ \oplus_{i \in I} \mathcal{X}_i ]^{\mathbb{Z}_+}$ denotes a sequence
indexed by $\mathbb{Z}_+$ whose components are elements of (vectors in) $\oplus_{i \in I} \mathcal{X}_i$.

\section{Problem Setup \& Statement}
\label{Sec:Setup}

\subsection{Setup}

Let $\mathcal{X}$ and $\mathcal{U}$ be finite dimensional vector spaces defined over algebraic field $\mathcal{F}$, 
with $dim(\mathcal{X})=n$ and $dim(\mathcal{U})=m$.
Consider the discrete-time dynamical system defined by the state transition equation
\begin{equation}
\label{Eq:DTsystem1}
x_{t+1} = A x_t + Bu_t
\end{equation}
where $x_t \in \mathcal{X}$, $u_t \in \mathcal{U}$, and $t \in \mathcal{T}$ for some index set $\mathcal{T}$.
$A: \mathcal{X} \rightarrow \mathcal{X}$ and $B: \mathcal{U} \rightarrow \mathcal{X}$ are given linear maps.
Consider also a non-negative cost function of the state, 
\begin{equation}
\label{Eq:CostFunction}
g: \mathcal{X} \rightarrow \mathbb{R}_+, \textrm{   with   } g(x)=0 \Leftrightarrow x=0.
\end{equation}

\begin{prob}
\label{Prob:1}
\textbf{(The DP problem)}
Consider system (\ref{Eq:DTsystem1}) and cost function (\ref{Eq:CostFunction}).
Given any initial state $x_0 \in \mathcal{X}$, we wish to find among all policies
$\pi : \mathcal{X} \rightarrow \mathcal{U}^{\mathcal{T}}$
 an optimal policy 
$\pi^*(x_0)$ that minimizes the additive cost
\begin{equation}
\label{Eq:AdditiveCost1}
J(x_0,\pi) = \sum_{t \in \overline{\mathcal{T}}} \alpha^t g(x_t)
\end{equation}
along the state trajectory $\{ x_t \}_{t \in \overline{\mathcal{T}}}$ 
starting at $x_0$ and evolving according to (\ref{Eq:DTsystem1}) under policy $\pi$.
We will consider two cases:
\begin{enumerate}
\item The \emph{finite horizon case}, where
$\mathcal{T}=\{0,1,\hdots,T-1\}$ and
$\overline{\mathcal{T}} = \{0,1,\hdots, T \}$
for some given positive integer $T$, and where $\alpha=1$.
\item The \emph{infinite horizon case}, where
$\mathcal{T} = \overline{\mathcal{T}} = \mathbb{Z}_+$,
and where $\alpha \in (0,1)$. \qed
\end{enumerate}
\end{prob}

\begin{remark}
An optimal policy for the DP problem always exists,  though it may not be unique in general.
\end{remark}

We denote the \emph{optimal cost} $J(x_0,\pi^*(x_0))$ in Problem \ref{Prob:1} by $J^*(x_0)$.
We use $(A,B,g,T)$ to denote the finite horizon version of Problem \ref{Prob:1}, 
and $(A,B,g, \alpha)$ to denote the infinite horizon version.

\subsection{The DP Solution}

The solution of Problem \ref{Prob:1} hinges on the {\it Principle of Optimality} \cite{BOOK:Bellma1957}. 
Indeed, for the finite horizon case, 
let $\mathcal{T}_{t}=\left\{t,t+1,\hdots,T-1\right\}$,
$\overline{\mathcal{T}}_{t}=\left\{t,t+1,\hdots,T\right\}$
 and define the {\it cost-to-go function} at time $t$ as
\begin{equation} 
\label{Eq:CostToGo}
J_t^*(x) = 
\left\{ \begin{array}{ccc}
\displaystyle \min_{\mathcal{U}^{\mathcal{T}_{t}}} \sum_{\tau \in \overline{\mathcal{T}}_t} g(x_{\tau}) & , &  \textrm{   when   }t \in \mathcal{T} \\
g(x) & , & \textrm{   when   } t=T
\end{array} \right. 
\end{equation}
for the system evolving according to (\ref{Eq:DTsystem1}) with $x_t=x$.
The principle of optimality can then be stated as 
\begin{equation}
\label{Eq:POfinite}
J_t^*(x) = g(x) + \min_{u \in \mathcal{U}} J_{t+1}^*(Ax+Bu) 
\end{equation}
for any $t \in \mathcal{T}$ and $x \in \mathcal{X}$.
It is used as the basis of the recursive DP algorithm \cite{BOOK:Bertse2005}
which solves these equations backwards in time for the optimal cost
$J^*(x_0) = J_0^*(x_0)$
and an optimal policy $\pi^*(x_0)$ consistent with an 
optimal controller $u^*: \mathcal{X} \times \mathcal{T} \rightarrow \mathcal{U}$ satisfying
\begin{equation}
\label{Eq:u*FH}
u^*(x,t) \in \argmin_{u \in \mathcal{U}} J^*_{t+1} (Ax + B u).
\end{equation}
Specifically, let $\{ x_t\}_{t \in \mathcal{T}}$ be the state trajectory of (\ref{Eq:DTsystem1})
starting from $x_0$ under policy $\pi^*(x_0)$.
The $t^{th}$ component of $\pi^*(x_0)$, which we will denote by $\pi_t^*(x_0)$,
equals $u^*(x_t,t)$.
For notational simplicity, we will sometimes write $u^*_t(x)$ to denote $u^*(x,t)$.

For the infinite horizon case, the principle of optimality gives rise to the \emph{Bellman equation}:
\begin{equation}
\label{Eq:POinfinite}
J^*(x) = g(x) + \alpha \min_{u \in \mathcal{U}} J^*(Ax+Bu).
\end{equation} 
The optimal cost can be determined via iterative methods such as value or policy iterations \cite{BOOK:Bertse2005},
and the resulting optimal policy corresponds to a state-feedback control law satisfying
\begin{equation}
\label{Eq:u*IH}
u^*(x) \in \argmin_{u \in \mathcal{U}} J^* (Ax + B u).
\end{equation}

\begin{remark}
While the DP problem is typically formulated in a setting where 
$\mathcal{X}= \mathbb{R}^n$ and $\mathcal{U} =\mathbb{R}^m$,
it is straightforward to verify that (\ref{Eq:POfinite}) and (\ref{Eq:POinfinite}) hold 
without modifications in the general setting considered here.
\end{remark}

\subsection{Problem Statement}

Consider a DP problem $(A,B,g,T)$ or $(A,B,g, \alpha)$ as formulated in Problem \ref{Prob:1}.
We will assume, without loss of generality, that $B$ is injective: 
That is, any matrix representation of $B$ has full column rank. 
Further assume that $\mathcal{X}$ can be decomposed into the direct sum of $A$-invariant subspaces, 
and that the cost function also exhibits an additive structure compatible with this decomposition. 
Specifically, 
\begin{equation}
\label{Eq:XDecomposition}
\mathcal{X} = \mathcal{X}_1 \oplus \hdots \oplus \mathcal{X}_r
\end{equation}
where 
\begin{equation}
\label{Eq:Xinvariance}
A \mathcal{X}_i \subseteq \mathcal{X}_i, 
\end{equation}
$i \in \mathcal{I}=\{1,\hdots,r\}$ for some $r>1$, and 
\begin{equation}
\label{Eq:gDecomposition}
g(x)  = g(x_1) + \cdots + g(x_r), \textrm{        } \forall x \in \mathcal{X}
\end{equation} 
where $x_i=\rho_i(x)$, the unique component of $x$ in $\mathcal{X}_i$.

We will be referring to a decomposition of $\mathcal{X}$ satisfying (\ref{Eq:XDecomposition}) and (\ref{Eq:Xinvariance})
simply as a \emph{decomposition of} $\mathcal{X}$ \emph{over} $A$,
and we will denote by $\mathcal{G}_s$ the set of all cost functions $g:\mathcal{X} \rightarrow \mathbb{R}_+$ that satisfy
(\ref{Eq:gDecomposition}).

\begin{remark}
A decomposition of $\mathcal{X}$ over $A$
arises naturally {\it regardless} of the underlying field $\mathcal{F}$.
Indeed, in the traditional setting where $\mathcal{X}=\mathbb{R}^n$ and the eigenvalues of $A$ are real, 
this decomposition is related to the generalized eigenspaces associated with the Jordan canonical form of $A$.
The more general case, associated with the rational canonical form, 
is addressed in the literature \cite{BOOK:Roman2008, BOOK:Lang2005}.
In general $r \ge 1$; the case of interest to us here is when $r>1$.
 \end{remark}
 
 {\bf We are fundamentally interested in understanding how, and under what conditions,
can this decomposition of the state-space and associated structure of the cost function be used 
to decompose Problem \ref{Prob:1} into a family of smaller DP problems that can be independently solved, 
with their solutions subsequently combined to yield that of Problem \ref{Prob:1}?}

\section{DP Decompositions}
\label{Sec:Decompositions}

We begin by formulating two families of DP problems consistent with the decomposition of $\mathcal{X}$ over $A$.
We then propose two corresponding notions of DP decomposition.

\subsection{Two Families of Problems}

Consider the subspaces $\mathcal{E}_i$ of $\mathcal{U}$, 
$i \in \mathcal{I}$, defined by
\begin{equation}
\mathcal{E}_i = \left\{u \in \mathcal{U}| Bu \in \mathcal{X}_i \right\}. \nonumber
\end{equation}
Note that $\mathcal{E}_i$ is simply the pre-image of $\mathcal{X}_i$ under $B$.
We are now ready to formulate the first family of DP problems:

\begin{prob}
\label{Prob:2}
Given system (\ref{Eq:DTsystem1}), cost function (\ref{Eq:CostFunction}), 
and a decomposition of $X$ over $A$ as in (\ref{Eq:XDecomposition}) and (\ref{Eq:Xinvariance}).
For each $i \in \mathcal{I}$,
consider the discrete-time dynamical system defined by the state transition equation
\begin{equation}
\label{Eq:DTsystem2}
x_{i,t+1} = A x_{i,t} + B \bar{u}_{i,t}
\end{equation}
where $x_{i,t} \in \mathcal{X}_i$ and $\bar{u}_{i,t} \in \mathcal{E}_i$.
Given any initial state $x_{i,0} \in \mathcal{X}_i$, we wish to find among all policies
$\overline{\pi}_i : \mathcal{X}_i \rightarrow \mathcal{E}_i^{\mathcal{T}}$
 an optimal policy 
$\overline{\pi}_i^*(x_{i,0})$ that minimizes the additive cost
\begin{equation}
\label{Eq:AdditiveCost2}
\overline{J}_i(x_{i,0},\overline{\pi}_i) = \sum_{t \in \overline{\mathcal{T}}} \alpha^t g(x_{i,t})
\end{equation}
along the state trajectory starting at $x_{i,0}$ and evolving according to (\ref{Eq:DTsystem2}) under policy $\overline{\pi}_i$.
We will consider two cases, the finite horizon and the infinite horizon ones,
with $\mathcal{T}$, $\overline{\mathcal{T}}$ and $\alpha$ defined as in Problem \ref{Prob:1}. \qed
\end{prob}

We denote the optimal cost $\overline{J}_i(x_{i,0},\overline{\pi}_i^*(x_{i,0}))$ of the 
$i^{th}$ subproblem of Problem \ref{Prob:2} by $\overline{J}_i^*(x_{i,0})$.
We use $\big\{ (A,B|_{\mathcal{E}_i},g,T) \big\}_{ i \in \mathcal{I}}$
to denote the entire family of finite horizon DP problems formulated in Problem \ref{Prob:2},
with $(A,B|_{\mathcal{E}_i},g,T)$ denoting the $i^{th}$ subproblem.
Similarly, we use $\big\{ (A,B|_{\mathcal{E}_i},g, \alpha) \big\}_{ i \in \mathcal{I}}$
to denote the entire family of infinite horizon DP problems formulated in Problem \ref{Prob:2},
with $(A,B|_{\mathcal{E}_i},g, \alpha)$ denoting the $i^{th}$ subproblem.

We formulate the second family of DP problems as follows:

\begin{prob}
\label{Prob:3}
Given system (\ref{Eq:DTsystem1}), cost function (\ref{Eq:CostFunction}), 
and a decomposition of $X$ over $A$ as in (\ref{Eq:XDecomposition}) and (\ref{Eq:Xinvariance}).
For each $i \in \mathcal{I}$,
consider the discrete-time dynamical system defined by the state transition equation
\begin{equation}
\label{Eq:DTsystem3}
x_{i,t+1} = A x_{i,t} + \rho_i \circ B u_{i,t}
\end{equation}
where $x_{i,t} \in \mathcal{X}_i$ and $u_{i,t} \in \mathcal{U}$.
Given any initial state $x_{i,0} \in \mathcal{X}$, we wish to find among all policies
$\pi_i : \mathcal{X}_i \rightarrow \mathcal{U}^{\mathcal{T}}$
 an optimal policy 
$\pi_i^*(x_{i,0})$ that minimizes the additive cost
\begin{equation}
\label{Eq:AdditiveCost3}
J_i(x_{i,0},\pi_i) = \sum_{t \in \overline{\mathcal{T}}} \alpha^t g(x_{i,t})
\end{equation}
along the state trajectory starting at $x_{i,0}$ and evolving according to (\ref{Eq:DTsystem3}) under policy $\pi_i$.
We will consider two cases, the finite horizon and the infinite horizon ones,
with $\mathcal{T}$, $\overline{\mathcal{T}}$ and $\alpha$ defined as in Problem \ref{Prob:1}. \qed
\end{prob}

We denote the optimal cost $J_i(x_{i,0},\pi_i^*(x_{i,0}))$ of the 
$i^{th}$ subproblem of Problem \ref{Prob:3} by $J_i^*(x_{i,0})$.
We use $\big\{ (A,\rho_i \circ B,g,T) \big\}_{ i \in \mathcal{I}}$
to denote the entire family of finite horizon DP problems formulated in Problem \ref{Prob:3},
with $(A,\rho_i \circ B,g,T)$ denoting the $i^{th}$ subproblem.
Similarly, we use $\big\{ (A,\rho_i \circ B,g, \alpha) \big\}_{ i \in \mathcal{I}}$
to denote the entire family of infinite horizon DP problems formulated in Problem \ref{Prob:3},
with $(A,\rho_i \circ B,g, \alpha)$ denoting the $i^{th}$ subproblem.

\subsection{Proposed Notions of Decomposition}

For each of the two families of DP problems, 
we now propose a corresponding notion of decomposition.

\begin{defnt} 
\label{Def:Decomp1}
Consider system (\ref{Eq:DTsystem1}),
a decomposition of $\mathcal{X}$ over $A$,
and a cost function (\ref{Eq:CostFunction}) satisfying $g \in \mathcal{G}_s$. 
The family $\big\{ (A,B|_{\mathcal{E}_i},g,T) \big\}_{ i \in \mathcal{I}}$ 
is a decomposition of $(A,B,g,T)$
if for any $x \in \mathcal{X}$, we have
\begin{equation}
\label{Eq:JDecomp1Cond}
J^*(x) = \sum_{i \in \mathcal{I}} \bar{J}^*_{i}(x_i)
\end{equation}
where $x_i =\rho_i(x)$,
and moreover, for any choice of optimal policies $\overline{\pi}_i^*(x_i)$, for $i \in \mathcal{I}$,
there exists an optimal policy $\pi^*(x)$ such that
\begin{equation}
\label{Eq:PiDecomp1Cond}
\pi^*(x) = \sum_{i \in \mathcal{I}} \overline{\pi}_i^* (x_i). 
\end{equation}       
Likewise, the family $\big\{ (A,B|_{\mathcal{E}_i},g, \alpha) \big\}_{ i \in \mathcal{I}}$ 
is a decomposition of $(A,B,g, \alpha)$ if for any $x \in \mathcal{X}$
we have (\ref{Eq:JDecomp1Cond}), 
and for any choice of optimal policies $\overline{\pi}_i^*(x_i)$, $i \in \mathcal{I}$,
there exists an optimal policy $\pi^*(x)$ such that (\ref{Eq:PiDecomp1Cond}) holds.
\end{defnt}

When $\big\{ (A,B|_{\mathcal{E}_i},g,T) \big\}_{ i \in \mathcal{I}}$ is a decomposition of $(A,B,g,T)$,
we can independently solve each of the smaller DP problems over time horizon $T$ and then simply add their solutions to obtain
the optimal policy and associated optimal cost of the original DP problem over the same time horizon.
Likewise when $\big\{ (A,B|_{\mathcal{E}_i},g, \alpha) \big\}_{ i \in \mathcal{I}}$ is a decomposition of $(A,B,g, \alpha)$.

\begin{defnt}
\label{Def:Decomp2}
Consider system (\ref{Eq:DTsystem1}),
a decomposition of $\mathcal{X}$ over $A$,
and a cost function (\ref{Eq:CostFunction}) satisfying $g \in \mathcal{G}_s$. 
The family $\big\{ (A,\rho_i \circ B ,g,T) \big\}_{ i \in \mathcal{I}}$ 
is a decomposition of $(A, B,g,T)$
if for any $x \in \mathcal{X}$, we have
\begin{equation}
\label{Eq:JDecomp2Cond}
J^*(x) = \sum_{i \in \mathcal{I}} J^*_{i}(x_i)
\end{equation}
where $x_i =\rho_i(x)$,
and moreover, for any choice of optimal policies $\pi_i^*(x_i)$, for $i \in \mathcal{I}$,
there exists an optimal policy $\pi^*(x)$ such that
\begin{equation}
\label{Eq:PiDecomp2Cond}
B \pi_t^*(x) = \sum_{i \in \mathcal{I}} \rho_i \circ B \pi_{i,t}^* (x_i), \forall t \in \mathcal{T}.
\end{equation}       
Likewise, the family $\big\{ (A,\rho_i \circ B ,g, \alpha) \big\}_{ i \in \mathcal{I}}$ 
is a decomposition of $(A, B,g, \alpha)$ if for any $x \in \mathcal{X}$
we have (\ref{Eq:JDecomp2Cond}), 
and for any choice of optimal policies $\pi_i^*(x_i)$, $i \in \mathcal{I}$,
there exists an optimal policy $\pi^*(x)$ such that (\ref{Eq:PiDecomp2Cond}) holds.
\end{defnt}

\color{black}
\section{Main Results}
\label{Sec:MainResults}

We begin by completely characterizing conditions under which the family of subproblems formulated in Problem \ref{Prob:2} 
is a decomposition of the original DP problem formulated in Problem \ref{Prob:1}.

\begin{lemma}
\label{Lemma:CondDecomp1FH}
Consider system (\ref{Eq:DTsystem1}),
a decomposition of $\mathcal{X}$ over $A$,
and a cost function (\ref{Eq:CostFunction}) satisfying $g \in \mathcal{G}_s$. 
$\big\{ (A,B|_{\mathcal{E}_i},g,T) \big\}_{ i \in \mathcal{I}}$ is a decomposition of $(A,B,g,T)$
iff
\begin{equation}
\label{Eq:NSCDecomp1FH}
\argmin_{u \in \mathcal{U}} J_{t+1}^*(Ax+Bu) \cap \left[\oplus_{i \in \mathcal{I}} \mathcal{E}_i \right] \neq \emptyset, \forall x \in \mathcal{X}, \forall t \in \mathcal{T}.
\end{equation} 
\end{lemma}

\begin{lemma}
\label{Lemma:CondDecomp1IH}
Consider system (\ref{Eq:DTsystem1}),
a decomposition of $\mathcal{X}$ over $A$,
and a cost function (\ref{Eq:CostFunction}) satisfying $g \in \mathcal{G}_s$. 
$\big\{ (A,B|_{\mathcal{E}_i},g, \alpha) \big\}_{ i \in \mathcal{I}}$ is a decomposition of $(A,B,g,\alpha)$
iff
\begin{equation}
\label{Eq:NSCDecomp1IH}
\argmin_{\pi \in \mathcal{U}^{\mathbb{Z}^+}} J(x,\pi) \cap [ \oplus_{i \in \mathcal{I}} \mathcal{E}_i]^{\mathbb{Z}^+} \neq \emptyset, \forall x \in \mathcal{X}.
\end{equation} 
\end{lemma}

Note that verifying the necessary and sufficient conditions established in Lemmas \ref{Lemma:CondDecomp1FH} 
and \ref{Lemma:CondDecomp1IH} effectively require solving the original DP problem.
Alternatively, we propose a readily verifiable sufficient condition to ensure that the decomposition exists.

\begin{thm} 
\label{Thm:SCDecomp1}
Consider system (\ref{Eq:DTsystem1}),
a decomposition of $\mathcal{X}$ over $A$,
and a cost function (\ref{Eq:CostFunction}) satisfying $g \in \mathcal{G}_s$. 
If 
\begin{equation}
\label{Eq:RangeBCondition}
\mathcal{R}(B) = \oplus_{i \in \mathcal{I}} \left[\mathcal{R}(B) \cap \mathcal{X}_i\right],
\end{equation} 
then $\big\{ (A,B|_{\mathcal{E}_i},g,T) \big\}_{ i \in \mathcal{I}}$ is a decomposition of $(A,B,g,T)$ for any choice of $T>0$,
and $\big\{ (A,B|_{\mathcal{E}_i},g, \alpha) \big\}_{ i \in \mathcal{I}}$ is a decomposition of $(A,B,g,\alpha)$.
\end{thm}

We also show that under certain conditions on the dynamics of (\ref{Eq:DTsystem1}), 
this condition becomes necessary as well as sufficient.

\begin{thm}
\label{Thm:NSCDecomp1}
Consider system (\ref{Eq:DTsystem1}),
a decomposition of $\mathcal{X}$ over $A$,
and a cost function (\ref{Eq:CostFunction}) satisfying $g \in \mathcal{G}_s$. 
When $A$ is invertible, the following three statements are equivalent:
\begin{enumerate}[(a)]
\item Condition (\ref{Eq:RangeBCondition}) holds.
\item $\big\{ (A,B|_{\mathcal{E}_i},g,T) \big\}_{ i \in \mathcal{I}}$ is a decomposition of $(A,B,g,T)$, for any choice of $T >0$.
\item $\big\{ (A,B|_{\mathcal{E}_i},g, \alpha) \big\}_{ i \in \mathcal{I}}$ is a decomposition of $(A,B,g,\alpha)$.
\end{enumerate}
\end{thm}

We next turn our attention to completely characterizing conditions under which the family of subproblems formulated in Problem \ref{Prob:3} 
is a decomposition of the original DP problem formulated in Problem \ref{Prob:1}.

\begin{lemma}
\label{Lemma:CondDecomp2FH}
Consider system (\ref{Eq:DTsystem1}),
a decomposition of $\mathcal{X}$ over $A$,
and a cost function (\ref{Eq:CostFunction}) satisfying $g \in \mathcal{G}_s$. 
$\big\{ (A,\rho_i \circ B ,g,T) \big\}_{ i \in \mathcal{I}}$ 
is a decomposition of $(A,B,g,T)$ iff
for any choice of optimal control laws of $\big\{ (A,\rho_i \circ B ,g,T) \big\}$, $i \in \mathcal{I}$,
there exists an optimal control law of $(A,B,g,T)$ such that
\begin{equation}
\label{Eq:NSCDecomp2FH}
\rho_i (Ax+Bu^*(x,t)) = Ax_i + \rho_i \circ B u^*_{i}(x_i,t), \forall x \in \mathcal{X}, \forall i \in \mathcal{I}, \forall t \in \mathcal{T}.
\end{equation}
This relation can be expressed in the commutative diagram

\[\begindc{\commdiag}[40]
\obj(0,1)[x1]{$\mathcal{X}$}
\obj(2,1)[x2]{$\mathcal{X}$}
\obj(0,0)[xb1]{$\mathcal{X}_i$}
\obj(2,0)[xb2]{$\mathcal{X}_i$}
\mor{x1}{x2}{$u^*_t$}
\mor{x2}{xb2}{$\rho_i$}
\mor{x1}{xb1}{$\rho_i$}
\mor{xb1}{xb2}{$u^*_{i,t}$}
\enddc\]
\end{lemma}

\begin{lemma}
\label{Lemma:CondDecomp2IH}
Consider system (\ref{Eq:DTsystem1}),
a decomposition of $\mathcal{X}$ over $A$,
and a cost function (\ref{Eq:CostFunction}) satisfying $g \in \mathcal{G}_s$. 
$\big\{ (A,\rho_i \circ B ,g,\alpha) \big\}_{ i \in \mathcal{I}}$ 
is a decomposition of $(A,B,g,\alpha)$ iff
for any choice of optimal control laws for $\big\{ (A,\rho_i \circ B ,g,\alpha) \big\}$, $i \in \mathcal{I}$,
there exists an optimal control law for $(A,B,g,\alpha)$ such that
\begin{equation}
\label{Eq:NSCDecomp2IH}
\rho_i (Ax+Bu^*(x)) = Ax_i + \rho_i \circ B u^*_{i}(x_i), \forall x \in \mathcal{X}, \forall i \in \mathcal{I}.
\end{equation}
This relation can be expressed in the commutative diagram

\[\begindc{\commdiag}[40]
\obj(0,1)[x1]{$\mathcal{X}$}
\obj(2,1)[x2]{$\mathcal{X}$}
\obj(0,0)[xb1]{$\mathcal{X}_i$}
\obj(2,0)[xb2]{$\mathcal{X}_i$}
\mor{x1}{x2}{$u^*$}
\mor{x2}{xb2}{$\rho_i$}
\mor{x1}{xb1}{$\rho_i$}
\mor{xb1}{xb2}{$u^*_{i}$}
\enddc\]
\end{lemma}

Finally, we establish the following hierarchy between the two proposed decompositions:

\begin{thm} 
\label{Thm:2implies1}
Consider system (\ref{Eq:DTsystem1}),
a decomposition of $\mathcal{X}$ over $A$,
and a cost function (\ref{Eq:CostFunction}) satisfying $g \in \mathcal{G}_s$. 
If $\big\{ (A,\rho_i \circ B,g,T) \big\}_{ i \in \mathcal{I}}$ is a decomposition of $(A,B,g,T)$, 
then $\big\{ (A,B|_{\mathcal{E}_i},g,T) \big\}_{ i \in \mathcal{I}}$ is also a decomposition of $(A,B,g,T)$.
Likewise, if $\big\{ (A,\rho_i \circ B,g,\alpha) \big\}_{ i \in \mathcal{I}}$ is a decomposition of $(A,B,g,\alpha)$, 
then $\big\{ (A,B|_{\mathcal{E}_i},g,\alpha) \big\}_{ i \in \mathcal{I}}$ is also a decomposition of $(A,B,g,\alpha)$.
\end{thm}

The converse statement in Theorem \ref{Thm:2implies1} does not necessarily hold, 
as we will see in Section \ref{Sec:Examples}.

\color{black}
\section{Derivation of Results}
\label{Sec:Derivation}

We begin by establishing some facts that will be helpful in our derivations:

\begin{prop} 
\label{lem:min-CDC}
Consider a system (\ref{Eq:DTsystem1}), a cost function (\ref{Eq:CostFunction}), 
and a decomposition of $\mathcal{X}$ over $A$.
Let $g \in \mathcal{G}_s$. 
Then for any $x_i \in \mathcal{X}_i$, we have
$\displaystyle \min_{u \in \mathcal{E}_i} g(Ax_i + Bu)=\min_{u \in \sum_{\j}\mathcal{E}_j} g(Ax_i + Bu)$.
\end{prop}

\begin{proof} 
We have
\begin{eqnarray*}
\min_{u \in \sum_{\j}\mathcal{E}_j} g(Ax_i + Bu)
& = &
\min_{u_j \in \mathcal{E}_j, \, j \in \mathcal{I}} g\left(Ax_i + B \left(\sum_{j \in \mathcal{I}} u_j\right) \right) \\
& = &
\min_{u_j \in \mathcal{E}_j, \, j \in \mathcal{I}}   \left[g(Ax_i+Bu_i)+\sum_{j \neq i}g\left( Bu_j \right) \right] \\
& = &
\min_{u_i \in \mathcal{E}_i} g(Ax_i+Bu_i) + \sum_{j \neq i} \min_{u_j \in \mathcal{E}_j} g\left(Bu_j \right) \\
& = & 
\min_{u_i \in \mathcal{E}_i} g(Ax_i+Bu_i)
\end{eqnarray*}
where the second equality follows from the assumption $g \in \mathcal{G}_s$,
the third equality follows from the fact that each $u_i$ independently affects one term of the summation,
and the fourth equality follows by noting that $g$ is non-negative and $g(0)=0$,
and by selecting $u_j = 0$ for $j \neq i$.
\end{proof} 

\begin{prop}
\label{Prop:Lemma1-1}
Consider a system (\ref{Eq:DTsystem1}), a cost function (\ref{Eq:CostFunction}), a decomposition of $\mathcal{X}$ over $A$,
and a DP problem $(A,B,g,T)$.
Let $g \in \mathcal{G}_s$, and assume that (\ref{Eq:NSCDecomp1FH}) holds for all $x \in \mathcal{X}$, $t \in \mathcal{T}$.
We have $J^*_t \in \mathcal{G}_s$, $\forall t \in \mathcal{T}$.
\end{prop}

\begin{proof}
By backwards induction on $T$.
For $t=T$ we have $J_T^*=g$ by definition, and thus $J_T^* \in \mathcal{G}_s$.
Now assume $J_{t}^* \in \mathcal{G}_s$,
and consider $x \in \mathcal{X}$ with $x_i$ denoting its unique component in $\mathcal{X}_i$.  
We can write 
\begin{eqnarray*}
J_{t-1}^*(x) 
& = & g(x) + \min_{u \in \mathcal{U}} J_t^*(Ax+Bu) \\
& = & g(x)+  \min_{u \in \sum_{i \in \mathcal{I}} \mathcal{E}_i} J_t^*(Ax+Bu) \\
& = & g(x) + \min_{u_i \in \mathcal{E}_i, i \in \mathcal{I}} J_t^*\left(\sum_{i\in \mathcal{I}} \left(Ax_i+Bu_i\right)\right) \\
& = & \sum_{i \in \mathcal{I}} g(x_i) + \min_{u_i \in \mathcal{E}_i, i \in \mathcal{I}} \sum_{i \in \mathcal{I}} J_t^*\left(Ax_i+Bu_i\right) \\
& = & \sum_{i\in \mathcal{I}} \left[g(x_i)+ \min_{u_i \in \mathcal{E}_i} J_t^*\left(Ax_i+Bu_i\right)\right] \\
& = & \sum_{i\in \mathcal{I}} \left[g(x_i)+ \min_{u \in \sum_{j \in \mathcal{I}} \mathcal{E}_j} J_t^*\left(Ax_i+Bu\right)\right] \\
& = & \sum_{i \in \mathcal{I}} \left[g(x_i)+ \min_{u \in \mathcal{U}} J_t^*\left(Ax_i+Bu\right) \right] \\
& = & \sum_{i \in \mathcal{I}} J^*_{t-1}(x_i) \\
\end{eqnarray*}
where the first and last equality follow from (\ref{Eq:POfinite}), 
the second equality follows from the assumption $J_{t}^* \in \mathcal{G}_s$,
the third from the decomposition of $\mathcal{X}$,
the fourth from the assumption that  $g, J_{t}^* \in \mathcal{G}_s$,
the fifth from the observation that each $u_i$ only affects one term in the summation,
the sixth from Proposition \ref{lem:min-CDC},
and the seventh from (\ref{Eq:NSCDecomp1FH}).
Since the choice of $x$ was arbitrary, we conclude that $J^*_{t-1} \in \mathcal{G}_s$.
\end{proof}

\begin{cor}
\label{Cor:Lemma1-1}
Consider a system (\ref{Eq:DTsystem1}), a cost function (\ref{Eq:CostFunction}), a decomposition of $\mathcal{X}$ over $A$,
and a DP problem $(A,B,g,T)$.
Let $g \in \mathcal{G}_s$, and assume that (\ref{Eq:NSCDecomp1FH}) holds for all $x \in \mathcal{X}$, $t \in \mathcal{T}$.
We have $J^* \in \mathcal{G}_s$.
\end{cor}

\begin{proof}
Follows immediately from Proposition \ref{Prop:Lemma1-1} by noting that $J^*(x) = J_0^*(x)$, $\forall x \in \mathcal{X}$.
\end{proof}

\begin{prop}
\label{Prop:Lemma1-2}
Consider a system (\ref{Eq:DTsystem1}), a cost function (\ref{Eq:CostFunction}), a decomposition of $\mathcal{X}$ over $A$,
a DP problem $(A,B,g,T)$, and a family of DP problems $\big\{ (A,B|_{\mathcal{E}_i},g,T) \big\}_{ i \in \mathcal{I}}$.
Let $g \in \mathcal{G}_s$, and assume that (\ref{Eq:NSCDecomp1FH}) holds for all $x \in \mathcal{X}$, $t \in \mathcal{T}$.
We have $J^*_t|_{\mathcal{X}_i}=\bar{J}^*_{i,t}$, $\forall t \in \mathcal{T}$, $\forall i \in \mathcal{I}$.
\end{prop}

\begin{proof}
By backwards induction on $t$.
For $t=T$, it follows from the definitions that $J_t^*(x_i) = g(x_i)= \overline{J}_{i,t}^*(x_i)$.
Now assume $J^*_t|_{\mathcal{X}_i}=\bar{J}^*_{i,t}$. 
We can write 
\begin{eqnarray*}
J^*_{t-1}(x_i)
& = & g(x_i) + \min_{u \in \mathcal{U}} J^*_{t}(Ax_i+Bu) \\
& = & g(x_i) + \min_{u \in \sum_{j \in \mathcal{I}} \mathcal{E}_j} J^*_{t}(Ax_i+Bu) \\
& = & g(x_i) + \min_{u_i \in \mathcal{E}_i} J^*_{t}(Ax_i+Bu_i) \\
& = & g(x_i) + \min_{u_i \in \mathcal{E}_i} \bar{J}^*_{i,t}(Ax_i+B_iu_i) \\
& = &  \bar{J}_{i,t-1}^*(x_i).
\end{eqnarray*} 
where the first and last equality follow from (\ref{Eq:POfinite}),
the second equality follows from (\ref{Eq:NSCDecomp1FH}),
the third from Propositions \ref{lem:min-CDC} and \ref{Prop:Lemma1-2},
and the fourth by our assumption.
The proof is completed by noting that the choices of $i$ and $x_i$ were arbitrary.
\end{proof}

\begin{cor}
\label{Cor:Lemma1-2}
Consider a system (\ref{Eq:DTsystem1}), a cost function (\ref{Eq:CostFunction}), a decomposition of $\mathcal{X}$ over $A$,
a DP problem $(A,B,g,T)$, and a family of DP problems $\big\{ (A,B|_{\mathcal{E}_i},g,T) \big\}_{ i \in \mathcal{I}}$.
Let $g \in \mathcal{G}_s$, and assume that (\ref{Eq:NSCDecomp1FH}) holds for all $x \in \mathcal{X}$, $t \in \mathcal{T}$.
We have $J^*|_{\mathcal{X}_i}=\bar{J}^*_{i}$.
\end{cor}

\begin{proof}
Follows immediately from Proposition \ref{Prop:Lemma1-2} by noting that $J^*(x) = J_0^*(x)$ and $J_i^*(x_i) = J_{i,0}^*(x_i)$.
\end{proof}

We are now ready to prove Lemmas \ref{Lemma:CondDecomp1FH} and \ref{Lemma:CondDecomp1IH}:

\begin{proof}[Proof of Lemma \ref{Lemma:CondDecomp1FH}]
Assume that $\big\{ (A,B|_{\mathcal{E}_i},g,T) \big\}_{ i \in \mathcal{I}}$ is a decomposition of $(A,B,g,T)$.
Then (\ref{Eq:NSCDecomp1FH}) follows immediately from 
(\ref{Eq:POfinite}) and (\ref{Eq:PiDecomp1Cond}).

Conversely, assume that (\ref{Eq:NSCDecomp1FH}) holds for all $x \in \mathcal{X}$, $t \in \mathcal{T}$.
From corollaries \ref{Cor:Lemma1-1} and \ref{Cor:Lemma1-2}, we have
\begin{displaymath}
J^*(x) = \sum_{i \in \mathcal{I}} J^*(x_i) = \sum_{i \in \mathcal{I}} \overline{J}^*_i(x_i)
\end{displaymath}
and thus (\ref{Eq:JDecomp1Cond}) holds.
Now for $x \in \mathcal{X}$ with $\rho_i(x)=x_i$,
pick a choice of optimal policy $\overline{\pi}^*_i$ for each $i \in \mathcal{I}$
and consider policy $\pi: \mathcal{X} \rightarrow \mathcal{U^{\mathcal{T}}}$
defined by $\pi(x) = \sum_{i \in \mathcal{I}} \overline{\pi}^*_i(x_i)$.
Let $\{x_t \} _{t \in \overline{\mathcal{T}}}$ be the state trajectory of (\ref{Eq:DTsystem1}) 
under policy $\pi$.
We have
\begin{eqnarray*}
J(x,\pi(x)) 
& = & g(x) + \sum_{t \in \mathcal{T}} g(Ax_t + B \pi_{t}(x))\\
& = & \sum_{i \in \mathcal{I}} g(x_i) + \sum_{t \in \mathcal{T}} g(A \sum_{i \in \mathcal{I}} x_{i,t }+ \sum_{i \in \mathcal{I}} B \overline{\pi}^*_{i,t}(x_i) ) \\
& = & \sum_{i \in \mathcal{I}} \big[ g(x_i) + \sum_{t \in \mathcal{T}} g(A x_{i,t }+ B \overline{\pi}^*_{i,t}(x_i) ) \big] \\
& = & \sum_{i \in \mathcal{I}} \overline{J}^*_i (x_i) \\
& = & J^*(x).
\end{eqnarray*}
where the first and fourth equality follow by definition,
the second equality follows from the choice of $\pi$ and the assumption that $g \in \mathcal{G}_s$,
the third follows from $g \in \mathcal{G}_s$,
and the fifth follows from Corollary \ref{Cor:Lemma1-2}.
Thus $\pi$ is indeed an optimal policy,
and (\ref{Eq:PiDecomp1Cond}) follows by noting that the choices of $x$ and $\overline{\pi}^*_i$ were arbitrary.
\end{proof}

\begin{remark}
It follows from condition (\ref{Eq:NSCDecomp1FH}) in Lemma \ref{Lemma:CondDecomp1FH} and the definition of the cost-to-go
function (\ref{Eq:CostToGo}) that if $\big\{ (A,B|_{\mathcal{E}_i},g,T) \big\}_{ i \in \mathcal{I}}$ is a decomposition of $(A,B,g,T)$
for some $T >0$, then $\big\{ (A,B|_{\mathcal{E}_i},g,T') \big\}_{ i \in \mathcal{I}}$ is a decomposition of $(A,B,g,T')$ 
for any choice of $T' <T$. The converse is not necessarily true.
\end{remark}

\color{black}
\begin{proof}[Proof of Lemma \ref{Lemma:CondDecomp1IH}]
Assume that $\big\{ (A,B|_{\mathcal{E}_i},g, \alpha) \big\}_{ i \in \mathcal{I}}$ is a decomposition of $(A,B,g,\alpha)$.
Then (\ref{Eq:NSCDecomp1IH}) follows immediately from (\ref{Eq:PiDecomp1Cond}) in Definition \ref{Def:Decomp1}.

Conversely, assume that (\ref{Eq:NSCDecomp1IH}) holds for all $x \in \mathcal{X}$.
Pick $x \in \mathcal{X}$ and $\pi^*(x) \in \bigoplus_{i \in \mathcal{I}} \mathcal{E}_i^{\mathbb{Z}^+}$. 
We can write $\pi^*(x)=\sum_{i \in \mathcal{I}} \pi_i(x)$, where $\pi_i(x) \in \mathcal{E}_i^{\mathbb{Z}^+}$, $\forall i \in \mathcal{I}$.
Now let $\{ x_t \}_{t \in \mathbb{Z}_+}$ be the state trajectory of (\ref{Eq:DTsystem1}) starting from
$x_0=x$ under policy $\pi^*(x)$.
We have
\begin{eqnarray*}
J^*(x) &= &\sum_{t=0}^{\infty} \alpha^t g(x_t) \\
& = & g(x) + \sum_{t=0}^{\infty} \alpha^{t+1} g\left(Ax_t+B\pi_t^*(x)\right) \\
& = & \sum_{i \in \mathcal{I}} g(x_i) + \sum_{t=0}^{\infty} \alpha^{t+1} \left[\sum_{i \in \mathcal{I}} g\left(Ax_{i,t}+B\pi_{i,t}(x) \right) \right],  \\
\end{eqnarray*}
where $\pi_{i,t}(x)$ denotes the $t^{th}$ component of $\pi_{i}(x)$ and the third equality follows from $g \in \mathcal{G}_s$.
We can thus write
\begin{eqnarray*}
J^*(x) & = & \sum_{i \in \mathcal{I}} \left[g(x_{i}) +\sum_{t=0}^{\infty} \alpha^{t+1} g\left(Ax_{i,t}+B\pi_{i,t}(x) \right) \right] \\
& = & \sum_{i \in \mathcal{I}} \bar{J}_i\left(x_{i},\pi_i(x)\right) \\
&\ge&  \sum_{i \in \mathcal{I}} \bar{J}_i^*(x_{i}).\\
\end{eqnarray*}
We also have, for any choice of any choice of optimal policies $\overline{\pi}_i^*(x_i)$, $i \in \mathcal{I}$,
\begin{eqnarray*}
\sum_{i \in \mathcal{I}} \bar{J}_i^*(x_{i}) & =  & \sum_{i \in \mathcal{I}} \bar{J}_i \left(x_{i},\bar{\pi}_{i}^*(x_{i}) \right) \\
& = & \sum_{i \in \mathcal{I}} \left[ g(x_{i}) + \sum_{t=0}^{\infty} \alpha^{t+1} g\left(Ax_{i,t}+B\bar{\pi}_{i,t}^*(x_i) \right) \right] \\
& = & g(x) + \sum_{t=0}^{\infty} \alpha^{t+1} \left[\sum_{i \in \mathcal{I}} g\left(Ax_{i,t}+B\bar{\pi}_{i,t}^*(x_i) \right) \right] \\
& = & g(x) + \sum_{t=0}^{\infty} \alpha^{t+1} g\left(Ax_t+B \left(\sum_{i \in \mathcal{I}}\bar{\pi}_{i,t}^*(x_i)\right)\right) \\
& =  & J \left(x, \sum_{i \in \mathcal{I}} \bar{\pi}_i^*(x_i) \right) \\
&\ge & J^*(x). 
\end{eqnarray*} 
Hence it follows that $J^*(x) = \sum_{i \in \mathcal{I}} \bar{J}_i^*(x_{i})$.
Since the choice of $x$ was arbitrary, (\ref{Eq:JDecomp1Cond}) holds.
Finally, since the inequality $\sum_{i \in \mathcal{I}} \bar{J}_i^*(x_{i}) \geq J^*(x)$ holds for any choice of $\overline{\pi}_i^*(x_i)$,
(\ref{Eq:PiDecomp1Cond}) also holds, 
and $\big\{ (A,B|_{\mathcal{E}_i},g, \alpha) \big\}_{ i \in \mathcal{I}}$ is a decomposition of $(A,B,g,\alpha)$.
\end{proof} 

The following result provides an intuitive characterization of condition (\ref{Eq:RangeBCondition}) that is also useful in proving Theorem \ref{Thm:SCDecomp1}.

\begin{prop} 
\label{Prop:RangeConditionInterpret}
$\mathcal{R}(B) = \oplus_{i \in \mathcal{I}} \left[\mathcal{R}(B) \cap \mathcal{X}_i \right] \Leftrightarrow \mathcal{U}=\sum_{i \in \mathcal{I}} \mathcal{E}_i$.
\end{prop}
\begin{proof} 
Assume $\mathcal{R}(B) = \oplus_{i \in \mathcal{I}} \left[\mathcal{R}(B) \cap \mathcal{X}_i \right]$ and pick $u \in \mathcal{U}$. 
Then $Bu=\sum_{i \in \mathcal{I}} b_i, \, b_i \in \mathcal{R}(B) \cap \mathcal{X}_i$. Since $b_i \in \mathcal{R}(B)$ there exists $u_i \in \mathcal{U}$ such that $b_i=Bu_i$. Since $b_i \in \mathcal{X}_i \Rightarrow u_i \in \mathcal{E}_i$. 
Thus $Bu = \sum_{i \in \mathcal{I}} Bu_i \Rightarrow u - \sum_{i \in \mathcal{I}} u_i \in \mathcal{N}(B)$. Since 
$\mathcal{N}(B) \subseteq \mathcal{E}_i$ for any $i \in \mathcal{I}$, $u \in \sum_{i \in \mathcal{I}} \mathcal{E}_i$, and the desired equality follows.

Conversely, assume $\mathcal{U}=\sum_{i \in \mathcal{I}} \mathcal{E}_i$ and pick $b \in \mathcal{R}(B)$. 
Then $b=Bu$ for some $u \in \mathcal{U}$. Now $u=\sum_{i \in \mathcal{I}} u_i, \, u_i \in \mathcal{E}_i$ and so $b = \sum_{i \in \mathcal{I}} Bu_i$ and $Bu_i \in \left[\mathcal{R}(B) \cap \mathcal{X}_i\right]$. Thus $\mathcal{R}(B) = \sum_{i \in \mathcal{I}} \left[\mathcal{R}(B) \cap \mathcal{X}_i \right]$ and this sum is direct.
\end{proof}

\begin{proof}[Proof of Theorem \ref{Thm:SCDecomp1}]
Assume that (\ref{Eq:RangeBCondition}) holds.
It follows from Proposition \ref{Prop:RangeConditionInterpret} that $\mathcal{U}=\sum_{i \in \mathcal{I}} \mathcal{E}_i$,
and thus (\ref{Eq:NSCDecomp1FH}) holds trivially for any choice of $T>0$, and (\ref{Eq:NSCDecomp1IH}) holds trivially.
\end{proof}

Let $\mathcal{V}$ be a subspace of $\mathcal{U}$ such that
\begin{equation}
\mathcal{U} = \oplus_{i \in \mathcal{I}} \mathcal{E}_i \oplus \mathcal{V}.
\end{equation}
In particular, when (\ref{Eq:RangeBCondition}) holds we have $\mathcal{V} = \{ 0 \}$.
We have the following observations:

\begin{prop}
\label{Prop:NCDecomp1FH}
Consider system (\ref{Eq:DTsystem1}),
a decomposition of $\mathcal{X}$ over $A$,
and a cost function (\ref{Eq:CostFunction}) satisfying $g \in \mathcal{G}_s$. 
If $\big\{ (A,B|_{\mathcal{E}_i},g,T) \big\}_{ i \in \mathcal{I}}$ is a decomposition of $(A,B,g,T)$,
then $A(\mathcal{X}) \cap B(\mathcal{V})=\left\{0 \right\}$.
\end{prop}

\begin{proof}
Assume $\big\{ (A,B|_{\mathcal{E}_i},g,T) \big\}_{ i \in \mathcal{I}}$ is a decomposition of $(A,B,g,T)$ and 
pick $\xi \in A(\mathcal{X}) \cap B(\mathcal{V})$.
We have $\xi = Ax_{\xi}=Bv_{\xi}$ for some $x_{\xi} \in \mathcal{X}$, $v_{\xi} \in \mathcal{V}$. 
By Lemma \ref{Lemma:CondDecomp1FH} it follows from (\ref{Eq:NSCDecomp1FH}) evaluated at $t=T-1$ that
there exists $u_{\xi} \in \oplus_i \mathcal{E}_i$ such that $\displaystyle \min_{u \in \mathcal{U}} g(Ax_{\xi} + Bu) = g(Ax_{\xi}+Bu_{\xi})$. 
But we have $\displaystyle \min_{u \in \mathcal{U}} g(Ax_{\xi} + Bu)=\min_{u \in \mathcal{U}} g(Bv_{\xi} + Bu)=0$,
which implies $g(Ax_{\xi}+Bu_{\xi})=g(Bv_{\xi}+Bu_{\xi})=0$. 
It follows from (\ref{Eq:CostFunction}) that $0=Bv_{\xi}+Bu_{\xi}=
B(v_{\xi}+u_{\xi})$, which implies that $v_{\xi}+u_{\xi}=0$, since $B$ is injective. 
The only possibility for the last relation to be true is $v_{\xi}=0$, from which it follows that $\xi = B v_{\xi} = 0$,
and thus $A(\mathcal{X}) \cap B(\mathcal{V})=\left\{0 \right\}$ indeed.
\end{proof}

\begin{prop} 
\label{Prop:J0}
Consider a DP problem $(A,B,g,T)$ or $(A,B,g,\alpha)$.
We have $J^*(x)=0 \Leftrightarrow x=0$.
\end{prop}

\begin{proof}
Since $g$ is non-negative, 
we have $J(x,\pi) \geq 0$, for all choices of $x$ and $\pi$.
It thus follows that $J^*(x) \geq 0$, for all $x \in \mathcal{X}$.
Now let $x=0$, and consider policy $\pi(0)$ defined by $\pi_t(0) \in \mathcal{N}(B)$ for all $t \in \mathcal{T}$.
The state trajectory of (\ref{Eq:DTsystem1}) starting at $x=0$ under policy $\pi$ satisfies $x_t=0$, $\forall t \in \overline{\mathcal{T}}$,
and its associated cost $J(0,\pi(0))=0$.
This policy thus achieves the minimum, and $J^*(0)=0$ indeed holds.
Conversely, assume $x \neq 0$:
We then have $g(x) >0$, and thus $J^*(x) \geq g(x) >0$.
\end{proof}

\begin{prop} 
\label{Prop:Pi0}
Consider a DP problem $(A,B,g,T)$ or $(A,B,g,\alpha)$.
We have $\pi_t^*(0) \in \mathcal{N}(B)$, $\forall t \in \mathcal{T}$.
\end{prop}

\begin{proof}
Follows immediately from the proof of Proposition \ref{Prop:J0}.
\end{proof}

\begin{prop}
\label{Prop:NCDecomp1IH}
Consider system (\ref{Eq:DTsystem1}),
a decomposition of $\mathcal{X}$ over $A$,
and a cost function (\ref{Eq:CostFunction}) satisfying $g \in \mathcal{G}_s$. 
If $\big\{ (A,B|_{\mathcal{E}_i},g,\alpha) \big\}_{ i \in \mathcal{I}}$ is a decomposition of $(A,B,g,\alpha)$,
then $A(\mathcal{X}) \cap B(\mathcal{V})=\left\{0 \right\}$.
\end{prop}

\begin{proof}
Assume $\big\{ (A,B|_{\mathcal{E}_i},g,\alpha) \big\}_{ i \in \mathcal{I}}$ is a decomposition of $(A,B,g,\alpha)$ and 
pick $\xi \in A(\mathcal{X}) \cap B(\mathcal{V})$.
We have $\xi = Ax_{\xi}=Bv_{\xi}$ for some $x_{\xi} \in \mathcal{X}$, $v_{\xi} \in \mathcal{V}$. 
It follows from (\ref{Eq:NSCDecomp1IH}) and (\ref{Eq:u*IH}) that 
there exists $u_{\xi} \in \oplus_i \mathcal{E}_i$ such that $\displaystyle \min_{u \in \mathcal{U}} J(Ax_{\xi} + Bu) = J(Ax_{\xi}+Bu_{\xi})$,
from Proposition \ref{Prop:J0} and an argument similar to that made in Proposition \ref{Prop:NCDecomp1FH} (omitted for brevity) 
that $0=Bv_{\xi}+Bu_{\xi}= B(v_{\xi}+u_{\xi})$, and thus $\xi = B v_{\xi} = 0$ and $A(\mathcal{X}) \cap B(\mathcal{V})=\left\{0 \right\}$.
\end{proof}

We are now ready to prove Theorem \ref{Thm:NSCDecomp1}:

\begin{proof}[Proof of Theorem \ref{Thm:NSCDecomp1}]
We have $(a) \Rightarrow (b)$ and $(a) \Rightarrow (c)$ by Theorem \ref{Thm:SCDecomp1}.
To show that $(b) \Rightarrow (a)$, assume that $\big\{ (A,B|_{\mathcal{E}_i},g,T) \big\}_{ i \in \mathcal{I}}$ is a decomposition of $(A,B,g,T)$.
It follows from Proposition \ref{Prop:NCDecomp1FH} that $A(\mathcal{X}) \cap B(\mathcal{V})=\left\{0 \right\}$.
Since $A$ in invertible by assumption, $A(\mathcal{X}) = \mathcal{X}$, and $B(\mathcal{V})=\{ 0 \}$.
Since $B$ is injective, we conclude that $\mathcal{V} =0$, and thus $\mathcal{U} = \oplus \mathcal{E}_i$,
and (\ref{Eq:RangeBCondition}) follows from Proposition \ref{Prop:RangeConditionInterpret}.
The proof $(c) \Rightarrow (a)$ similarly follows from Proposition \ref{Prop:NCDecomp1IH}.
\end{proof}

We now turn our attention to the second family of DP problems and their associated notion of decomposition.
We begin by establishing the following intermediate result:

\begin{prop}
\label{Prop:EquivCondsFH}
Consider system (\ref{Eq:DTsystem1}),
a decomposition of $\mathcal{X}$ over $A$,
a cost function (\ref{Eq:CostFunction}) satisfying $g \in \mathcal{G}_s$,
a DP problem $(A,B,g,T)$ and a family $\big\{ (A,\rho_i \circ B ,g,\alpha) \big\}_{ i \in \mathcal{I}}$.
The following two statements are equivalent:
\begin{enumerate}[(a)]
\item For any choice of optimal policies $\pi_i^*$, $i \in \mathcal{I}$, there exists an optimal policy $\pi^*$ such that (\ref{Eq:PiDecomp2Cond}) holds.
\item For any choice of optimal control laws of $\big\{ (A,\rho_i \circ B ,g,T) \big\}$, $i \in \mathcal{I}$,
there exists an optimal control law of $(A,B,g,T)$ such that (\ref{Eq:NSCDecomp2FH}) holds.
\end{enumerate}
\end{prop}

\begin{proof}
Pick $x \in \mathcal{X}$ and a choice of optimal policies $\pi_i^*(x_i)$, $i \in \mathcal{I}$.
Let $\{ x_{i,t} \}_{t \in \overline{\mathcal{T}}}$ be the state trajectory of the $i^{th}$ subsystem under 
policy $\pi_i^*(x_i)$, and let $u_i^*(x_{i,t},t) = \pi^*_{i,t}(x_i)$ be the corresponding optimal control laws.

To show $(a) \Rightarrow (b)$, let $\pi^*$ be an optimal policy such that (\ref{Eq:PiDecomp2Cond}) holds,
let $\{ x_t \}_{t \in \overline{\mathcal{T}}}$ be the state trajectory of (\ref{Eq:DTsystem1}) under this policy,
with $u^*(x_t,t) = \pi^*_t(x)$.
We have
\begin{eqnarray*}
B \pi_t^*(x) =  \sum_{i \in \mathcal{I}} \rho_i \circ B \pi_{i,t}^* (x_i), \forall t \in \mathcal{T} 
& \Leftrightarrow & Ax + B \pi_t^*(x) = Ax + \sum_{i \in \mathcal{I}} \rho_i \circ B \pi_{i,t}^* (x_i), \forall t \in \mathcal{T} \\
& \Leftrightarrow & Ax + B u^*(x,t) = \sum_{i \in \mathcal{I}} Ax_i + \sum_{i \in \mathcal{I}} \rho_i \circ B u_{i}^* (x_i,t), \forall t \in \mathcal{T} \\
& \Leftrightarrow & Ax + B u^*(x,t) = \sum_{i \in \mathcal{I}} \big[ Ax_i +  \rho_i \circ B u_{i}^* (x_i,t) \big], \forall t \in \mathcal{T} \\
& \Leftrightarrow & \rho_i (Ax + B u^*(x,t)) = Ax_i +  \rho_i \circ B u_{i}^* (x_i,t), \forall t \in \mathcal{T}, \forall i \in \mathcal{I}. 
\end{eqnarray*}
What is left is to note that the choice of $x$ was arbitrary.

To show $(b) \Rightarrow (a)$, 
let $u^*$ be an optimal control law satisfying (\ref{Eq:NSCDecomp2FH}) and consider the policy $\pi^*$
defined by  $\pi^*_t(x) = u^*(x_t,t)$. 
By the above equivalence, we have that $\pi^*$ satisfies (\ref{Eq:PiDecomp2Cond}).
\end{proof}

We are now ready to prove Lemma \ref{Lemma:CondDecomp2FH}:  

\begin{proof}[Proof of Lemma \ref{Lemma:CondDecomp2FH}]
Having established in Proposition \ref{Prop:EquivCondsFH} the equivalence between conditions 
(\ref{Eq:PiDecomp2Cond}) and (\ref{Eq:NSCDecomp2FH}),
what is left is to show that (\ref{Eq:NSCDecomp2FH}) $\Rightarrow$ (\ref{Eq:JDecomp2Cond}).

Assume that (\ref{Eq:NSCDecomp2FH}) holds,
let $\{ x_{i,t} \}_{t \in \overline{\mathcal{T}}}$ be the state trajectory of the $i^{th}$ subsystem under
the optimal control law $u_i^*(x_{i,t},t)$, $i \in \mathcal{I}$,
and let $\{ x_t \}_{t \in \overline{\mathcal{T}}}$ be the state trajectory of (\ref{Eq:DTsystem1}) under 
the corresponding optimal control law $u^*(x_{t},t)$.
For any $x \in \mathcal{X}$, we have
\begin{eqnarray*}
J^*(x)
& = & g(x) + \sum_{t \in \mathcal{T}} g(Ax_t+ B u^*(x,t)) \\
& = & g(x) +  \sum_{t \in \mathcal{T}} g\big( \sum_{i \in \mathcal{I}} \rho_i ( Ax_t+ B u^*(x,t) ) \big) \\
& = & g(x) +  \sum_{t \in \mathcal{T}} g\big( \sum_{i \in \mathcal{I}} ( Ax_{i,t} + \rho_i \circ B u_i^*(x_{i,t},t) ) \big) \\
& = & \sum_{i \in \mathcal{I}} g(x_i) +  \sum_{t \in \mathcal{T}} \sum_{i \in \mathcal{I}} g\big(( Ax_{i,t} + \rho_i \circ B u_i^*(x_{i,t},t) ) \big) \\
& = & \sum_{i \in \mathcal{I}} \big[ g(x_i) +  \sum_{t \in \mathcal{T}} g ( Ax_{i,t} + \rho_i \circ B u_i^*(x_{i,t},t) ) \big] \\
& = & \sum_{i \in \mathcal{I}} J_i^*(x_i)
\end{eqnarray*}
where the first equality follows by definition, 
the second from the definition of $\rho_i$,
the third from (\ref{Eq:NSCDecomp2FH}),
the fourth from $g \in \mathcal{G}_s$,
and the fifth and sixth by definition.
\end{proof}

We can establish an analogous result for the infinite horizon setting and use it in proving Lemma \ref{Lemma:CondDecomp2IH}:

\begin{prop}
\label{Prop:EquivCondsIH}
Consider system (\ref{Eq:DTsystem1}),
a decomposition of $\mathcal{X}$ over $A$,
a cost function (\ref{Eq:CostFunction}) satisfying $g \in \mathcal{G}_s$,
a DP problem $(A,B,g,\alpha)$ and a family $\big\{ (A,\rho_i \circ B ,g,\alpha) \big\}_{ i \in \mathcal{I}}$.
The following two statements are equivalent:
\begin{enumerate}[(a)]
\item For any choice of optimal policies $\pi_i^*$, $i \in \mathcal{I}$,
there exists an optimal policy $\pi^*$ such that (\ref{Eq:PiDecomp2Cond}) holds.
\item For any choice of optimal control laws of $\big\{ (A,\rho_i \circ B ,g,\alpha) \big\}$, $i \in \mathcal{I}$,
there exists an optimal control law of $(A,B,g,\alpha)$ such that (\ref{Eq:NSCDecomp2IH}) holds.
\end{enumerate}
\end{prop}

\begin{proof}
The proof is similar to that of Proposition \ref{Prop:EquivCondsFH} and is omitted for brevity.
\end{proof}

\begin{proof}[Proof of Lemma \ref{Lemma:CondDecomp2IH}]
Having established in Proposition \ref{Prop:EquivCondsIH} the equivalence between conditions 
(\ref{Eq:PiDecomp2Cond}) and (\ref{Eq:NSCDecomp2IH}),
what is left is to show that (\ref{Eq:NSCDecomp2IH}) $\Rightarrow$ (\ref{Eq:JDecomp2Cond}).

Assume that (\ref{Eq:NSCDecomp2IH}) holds,
let $\{ x_{i,t} \}_{t \in \overline{\mathcal{T}}}$ be the state trajectory of the $i^{th}$ subsystem under 
the optimal control law $u_i^*(x_{i,t})$, $i \in \mathcal{I}$,
and let $\{ x_t \}_{t \in \overline{\mathcal{T}}}$ be the state trajectory of (\ref{Eq:DTsystem1}) under 
the corresponding optimal control law $u^*(x_{t})$.
For any $x \in \mathcal{X}$, we have
\begin{eqnarray*}
J^*(x)
& = & g(x) + \sum_{t =0}^{\infty} \alpha^{t+1} g(Ax_t+ B u^*(x_t)) \\
& = & g(x) +  \sum_{t=0}^{\infty} \alpha^{t+1} g\big( \sum_{i \in \mathcal{I}} \rho_i ( Ax_t+ B u^*(x_t ) \big) \\
& = & g(x) +  \sum_{t=0}^{\infty} \alpha^{t+1} g\big( \sum_{i \in \mathcal{I}} ( Ax_{i,t} + \rho_i \circ B u_i^*(x_{i,t}) ) \big) \\
& = & \sum_{i \in \mathcal{I}} g(x_i) +  \sum_{t =0}^{\infty} \alpha^{t+1} \sum_{i \in \mathcal{I}} g\big(( Ax_{i,t} + \rho_i \circ B u_i^*(x_{i,t}) ) \big) \\
& = & \sum_{i \in \mathcal{I}} \big[ g(x_i) +  \sum_{t=0}^{\infty} \alpha^{t+1} g ( Ax_{i,t} + \rho_i \circ B u_i^*(x_{i,t}) ) \big] \\
& = & \sum_{i \in \mathcal{I}} J_i^*(x_i)
\end{eqnarray*}
where the first equality follows by definition, 
the second from the definition of $\rho_i$,
the third from (\ref{Eq:NSCDecomp2IH}),
the fourth from $g \in \mathcal{G}_s$,
and the fifth and sixth by definition.
\end{proof}

\begin{proof}[Proof of Theorem \ref{Thm:2implies1}]
Assume that $\big\{ (A,\rho_i \circ B,g,T) \big\}_{ i \in \mathcal{I}}$ is a decomposition of $(A,B,g,T)$.
By Definition \ref{Def:Decomp2}, (\ref{Eq:PiDecomp2Cond}) holds for all $x \in \mathcal{X}$.
In particular, for $z \in \mathcal{X}_i$, we have the following $\forall t \in \mathcal{T}$:
\begin{eqnarray*}
B \pi_t^*(z) & = & \sum_{i \in \mathcal{I}} \rho_i \circ B \pi_{i,t}^* (z_i) \\
& = & \rho_i \circ B \pi_{i,t}^* (z) +  \sum_{j \in \mathcal{I}, j \neq i} \rho_j \circ B \pi_{j,t}^* (0) \\
& = & \rho_i \circ B \pi_{i,t}^* (z) 
\end{eqnarray*}
where the second equality follows from the fact that $z \in \mathcal{X}_i$ and the third equality follows
from Proposition \ref{Prop:Pi0} with $\rho_j \circ B$ replacing $B$.
We thus conclude that for $z \in \mathcal{X}_i$,
$B \pi^*_t(z) \in \mathcal{E}_i$, $\forall t \in \mathcal{T}$.

Now pick an $x \in \mathcal{X}$, a choice of optimal policies $\pi_i^*$, $i \in \mathcal{I}$ and an 
optimal policy $\pi^*$ satisfying (\ref{Eq:PiDecomp2Cond}).
Let $\{ x_{i,t}\}_{t\in \overline{T}}$ and $\{ x_t\}_{t \in \overline{\mathcal{T}}}$ be the corresponding state trajectories.
By (\ref{Eq:PiDecomp2Cond}), we have for all $t \in \mathcal{T}$
\begin{eqnarray*}
B u^*(x_t,t) & = & B \pi_t^*(x) \\
& = & \sum_{i \in \mathcal{I}} \rho_i \circ B \pi_{i,t}^* (x_i) \\
& = & \sum_{i \in \mathcal{I}} B \pi_{t}^* (x_i)  \\
& = & \sum_{i \in \mathcal{I}} B u_{i}^* (x_{i,t},t) 
\end{eqnarray*}
from which we have $u^*(x_t,t) - \sum_{i \in \mathcal{I}} u_{i}^* (x_{i,t},t) \in \mathcal{N}(B)$,
and hence $u^*(x_t,t) = \sum_{i \in \mathcal{I}} u_{i}^* (x_{i,t},t)$ since $B$ is injective.
It thus follows that $u^*(x_t,t) \in \sum_{i \in \mathcal{I}} \mathcal{E}_i$, for all choices of $x_t$ and $t$,
and thus (\ref{Eq:NSCDecomp1FH}) holds and $\big\{ (A,B|_{\mathcal{E}_i},g, \alpha) \big\}_{ i \in \mathcal{I}}$ 
is a decomposition of $(A,B,g,T)$.

The proof for the infinite horizon case is similar, and is thus omitted for brevity.
\end{proof}

\section{Illustrative Examples}
\label{Sec:Examples}

Our first example is a familiar instance of dynamic programming:

\begin{example}
Consider the case where $\mathcal{X}=\mathbb{R}^n, \mathcal{U}=\mathbb{R}^m$, 
and the dynamics $x_{t+1}=Ax_t+Bu_t$ with $A \in \mathbb{R}^{n \times n}, B \in \mathbb{R}^{n \times m}$ over a finite horizon $t \in \left\{0, 1, \hdots, T\right\}, T \in \mathbb{Z}^+$. Let $g(x)=x^T Px$ where $P$ is positive-definite.
Then $J_t^*(x) = x^T K_t x$ with $K_T=P$ and $K_t$ given by the backward algebraic Riccati recursion (assuming that $B$ has full column rank)
\begin{displaymath}
K_t = P + A^T K_{t+1} A - A^T K_{t+1} B (B^T K_{t+1} B)^{-1} B^T K_{t+1} A
\end{displaymath} 
and the optimal controller is given by
\begin{displaymath}
u_t^*(x) = -(B^T K_{t} B)^{-1} B^T K_{t} Ax.
\end{displaymath}
Now let $A=S J S^{-1}$ be the canonical Jordan decomposition of $A$ and denote by $S_i$ the submatrix of $S$ consisting of the columns of $S$ corresponding to the $i^{th}$ Jordan block of $J$. Define $\mathcal{X}_i=\mathcal{R}(S_i)$. Then by construction $\mathcal{X}_i$ is $A$-invariant and  $\mathbb{R}^n=\oplus_{i \in \mathcal{I}} \mathcal{X}_i$. Assume in addition that the conditions of Thm. \ref{cor:rangeB} are true. This implies that the subspaces $\mathcal{X}_i$ must be $K_t$-orthogonal in order for $J_t^*$ to split, i.e.
$x_i^T K_t x_j=0$, whenever $x_i \in \mathcal{X}_i, x_j \in \mathcal{X}_j, i \neq j$. Note also that $\mathbb{R}^m=\oplus_{i \in \mathcal{I}} \mathcal{E}_i$.
Consider the representation of $A,P$ on a basis of $\mathbb{R}^n$
given by the union of basis of each of the subspaces $\mathcal{X}_i$; then $P,A$ will be block diagonal. Moreover choosing as a basis of $\mathbb{R}^m$ the union of basis of the subspaces $\mathcal{E}_i$ and representing the image of $B$ using the above mentioned basis of $\mathbb{R}^n$, yields $B$ in a block diagonal form as well (even though $B$ need not be square). Then it is seen that the algebraic Riccati recursion becomes block diagonal, the block recursions representing the Riccati recursions corresponding to the subsystems $(A_i,B_i)$. Finally, the optimal controller itself is diagonal, each of its entries giving an optimal controller for the corresponding subsystem.
\end{example}

Our next example considers a finite state system and illustrates that (\ref{Eq:RangeBCondition}) 
is indeed sufficient, but not necessary in general, for a decomposition to exist:

\begin{example}
Let $\mathcal{X}=\left(\mathbb{Z}_3\right)^{3}, \mathcal{U} = \left(\mathbb{Z}_3 \right)^{2}$ and consider the system
\begin{displaymath}
x_{t+1}= \left[\begin{array}{ccc} 1 & 1 & 0 \\ 0 & 2 & 0 \\ 0 & 0 & 1 \end{array}\right] x_t + \left[\begin{array}{cc} 1 & 0 \\ 1 & 1 \\ 0 & 1 \end{array} \right] u_t
\end{displaymath}
The invariant subspaces are $\mathcal{X}_1=<\left[1 \, \, 0 \, \, 0 \right]^T>, \mathcal{X}_2=<\left[1 \, \, 1 \, \, 0 \right]^T>, \mathcal{X}_3=<\left[ 0 \, \, 0 \, \, 1 \right]^T>$. 
Note that $\mathcal{R}(B) \cap \mathcal{X}_1=\mathcal{R}(B) \cap \mathcal{X}_3=\left\{0\right\}, \mathcal{R}(B) \cap \mathcal{X}_2=\mathcal{X}_2$. 
The three subsystems are 
\begin{align} \nonumber
(A_1,B_1) &= \left(\left[\begin{array}{ccc} 1 & 0 & 0 \\ 0 & 0 & 0 \\ 0 & 0 & 0 \end{array}\right],\left[\begin{array}{cc} 0 & 0 \\ 0 & 0 \\ 0 & 0 \end{array} \right] \right) \nonumber \\
(A_2,B_2) &= \left(\left[\begin{array}{ccc} 0 & 2 & 0 \\ 0 & 2 & 0 \\ 0 & 0 & 0 \end{array}\right], \left[\begin{array}{cc} 1 & 0 \\ 1 & 0 \\ 0 & 0 \end{array} \right]\right) \nonumber \\
(A_3,B_3) &= \left(\left[\begin{array}{ccc} 0 & 0 & 0 \\ 0 & 0 & 0 \\ 0 & 0 & 1 \end{array}\right],\left[\begin{array}{cc} 0 & 0 \\ 0 & 0 \\ 0 & 0 \end{array} \right] \right). \nonumber \\
\end{align} 
Since $\left(\mathbb{Z}_3 \right)^{3}=\mathcal{X}_1 \oplus \mathcal{X}_2 \oplus \mathcal{X}_3$, every element of $\left(\mathbb{Z}_3 \right)^{3}$ 
can be written as a unique linear combination of $\left[1 \, \, 0 \, \, 0 \right]^T, \left[1 \, \, 1 \, \, 0 \right]^T, \left[0 \, \, 0 \, \, 1 \right]^T$.
In particular for any $\alpha_1,\alpha_2,\alpha_3 \in \mathbb{Z}_3$ we have that
\begin{displaymath}
\left[\begin{array}{c} \alpha_1 \\ \alpha_2 \\ \alpha_3 \end{array} \right]=(\alpha_1+2\alpha_2)\left[\begin{array}{c} 1 \\ 0 \\ 0 \end{array} \right]+
\alpha_2 \left[\begin{array}{c} 1 \\ 1 \\ 0 \end{array} \right] + \alpha_3 \left[\begin{array}{c} 0 \\ 0 \\ 1 \end{array} \right]. 
\end{displaymath}
Now consider a cost function $g: \left(\mathbb{Z}_3 \right)^{3} \rightarrow \mathbb{R}^+$ with the property that 
$g\left(\left[\begin{array}{c} \alpha_1 \\ \alpha_2 \\ \alpha_3 \end{array} \right]\right) = g\left(\alpha_2 \left[\begin{array}{c} 1 \\ 1 \\ 0 \end{array} \right]\right)$, i.e. $g$ penalizes only $\mathcal{X}_2$. 
Consider also a finite horizon $T=1$. 
Then the optimal controller corresponding to state $\left[x_1 \,  x_2 \, x_3\right]^T$ is given as the solution to the problem
\begin{displaymath}
\min_{u_1,u_2 \in \mathbb{Z}_3} g\left(\left[\begin{array}{ccc} 1 & 1 & 0 \\ 0 & 2 & 0 \\ 0 & 0 & 1 \end{array}\right]\left[\begin{array}{c} x_1 \\ x_2 \\ x_3 \end{array} \right]+\left[\begin{array}{cc} 1 & 0 \\ 1 & 1 \\ 0 & 1 \end{array} \right]\left[\begin{array}{c} u_1 \\ u_2  \end{array} \right] \right) 
\end{displaymath} 
which can equivalently be written as
\begin{displaymath}
\min_{u_1,u_2 \in \mathbb{Z}_3} g\left(\left[\begin{array}{ccc} 1 & 1 & 0 \\ 0 & 2 & 0 \\ 0 & 0 & 1 \end{array}\right]\left[\begin{array}{c} x_1 \\ x_2 \\ x_3 \end{array} \right]+ \right.
\left. \left[\begin{array}{cccc} 1 & 1 & 1 &0 \\ 1 & 0 & 1 &  0 \\ 0 & 0 & 0 & 1 \end{array} \right] \left[\begin{array}{c} u_1 \\ 2u_2 \\ u_2 \\ u_2 \end{array} \right]\right) 
\end{displaymath}
or equivalently using the property of $g$
\begin{displaymath}
\min_{u_1,u_2 \in \mathbb{Z}_3} g\left(\left[\begin{array}{ccc} 1 & 1 & 0 \\ 0 & 2 & 0 \\ 0 & 0 & 1 \end{array}\right]\left[\begin{array}{c} x_2 \\ x_2 \\ 0 \end{array} \right]+(u_1+u_2)\left[\begin{array}{c} 1  \\ 1  \\ 0 \end{array} \right] \right)  
\end{displaymath} 
which is equivalent to
\begin{displaymath}
\min_{u_1 \in \mathbb{Z}_3} g\left(\left[\begin{array}{ccc} 0 & 2 & 0 \\ 0 & 2 & 0 \\ 0 & 0 & 0 \end{array}\right]\left[\begin{array}{c} x_2 \\ x_2 \\ 0 \end{array} \right]+u_1\left[\begin{array}{c} 1 \\ 1 \\ 0  \end{array} \right]  \right) 
\end{displaymath}
or equivalently
\begin{equation}
\min_{u_1 \in \mathbb{Z}_3} g\left(\left[\begin{array}{ccc} 0 & 2 & 0 \\ 0 & 2 & 0 \\ 0 & 0 & 0 \end{array}\right]\left[\begin{array}{c} x_2 \\ x_2 \\ 0 \end{array} \right]+\left[\begin{array}{cc} 1 & 0 \\ 1 & 0  \\ 0 & 0 \end{array} \right]\left[\begin{array}{c} u_1 \\ 0 \end{array} \right]\right) \nonumber
\end{equation}
the latter being precisely the problem giving the optimal controller of subsystem 2. 
Consequently, if $\left[u_1^* \, 0 \right]^T$ is an optimal controller for subsystem 2 corresponding to state $\left[ x_2 \, x_2 \, 0 \right]^T$, 
then $\left[u_1^* \, 0 \right]^T$ is also an optimal controller for the original system corresponding to state $\left[x_1 \, x_2 \, x_3 \right]^T$ for any $x_1, x_3 \in \mathbb{Z}_3$.
\end{example}

Our last example demonstrates that one notion of decomposition implies the other, but not vice-versa:

\begin{example}
Consider the reachable linear system over the real number field with
\begin{displaymath}
A =\left[
\begin{array}{ccc}
1 & 0 & 0 \\
0 & 1 & 0 \\
0 & 0 & 0
\end{array}
\right], \, \, \, 
B = \left[
\begin{array}{cc}
1 & 1 \\
0 & 1 \\
0 & 1
\end{array}
\right].
\end{displaymath}
Take the invariant subspaces to be $\mathcal{X}_1=<e_1>_{\mathbb{R}}, \mathcal{X}_2=<e_2>_{\mathbb{R}}, \mathcal{X}_3=<e_3>_{\mathbb{R}}$ where $e_i$ is the standard unit vector of $\mathbb{R}^{3 \times 1}$. Note that $\mathcal{E}_1=<\left[\begin{array}{c} 1 \\ 0 \end{array}\right]>_{\mathbb{R}}, \mathcal{E}_2=\mathcal{E}_3=\left\{0\right\}, \mathcal{V}=<\left[0 \, \, \, 1\right]'>_{\mathbb{R}}$. Let $h : \mathbb{R} \rightarrow \mathbb{R}^+$ be defined by
\begin{displaymath}
h(\xi) = \left\{
\begin{array}{cc}
1, & \xi \neq 0 \\
0, & \xi=0
\end{array}
\right.
\end{displaymath}
and for any state $x=(x_1,x_2,x_3)$ define a cost function of the state
\begin{displaymath}
g(x) = h(x_1)+h(x_2)+h(x_3). \nonumber
\end{displaymath}
Finally consider for simplicity a finite horizon $T=1$.
The systems of $\big\{ (A,B|_{\mathcal{E}_i},g,\alpha) \big\}_{ i \in \mathcal{I}}$ and their
 corresponding solutions are
\begin{align}
A|_{\mathcal{X}_1} =\left[
\begin{array}{ccc}
1 & 0 & 0 \\
0 & 0 & 0 \\
0 & 0 & 0
\end{array}
\right], \, &
B|_{\mathcal{E}_1} = \left[
\begin{array}{cc}
1 & 0 \\
0 & 0 \\
0 & 0
\end{array}
\right] \nonumber \\
A|_{\mathcal{X}_2}=\left[
\begin{array}{ccc}
0 & 0 & 0 \\
0 & 1 & 0 \\
0 & 0 & 0
\end{array}
\right], \, &
B|_{\mathcal{E}_2} = \left[
\begin{array}{cc}
0 & 0 \\
0 & 0 \\
0 & 0
\end{array}
\right] \nonumber \\
A|_{\mathcal{X}_3} =\left[
\begin{array}{ccc}
0 & 0 & 0 \\
0 & 0 & 0 \\
0 & 0 & 0
\end{array}
\right], \, &
B|_{\mathcal{E}_3} = \left[
\begin{array}{cc}
0 & 0 \\
0 & 0 \\
0 & 0
\end{array}
\right] \nonumber
\end{align} 
\begin{align}
\bar{u}^*_{1,T-1}(x_1,0,0) &= [-x_1 \, \, 0]', & \bar{J}^*_{1,T-1}(x_1,0,0) &=h(x_1) \nonumber \\
\bar{u}^*_{2,T-1}(0,x_2,0) &= [0 \, \, 0]', & \bar{J}^*_{2,T-1}(0,x_2,0)&=2h(x_2) \nonumber \\
\bar{u}^*_{3,T-1}(0,0,x_3) &= [0 \, \, 0]', & \bar{J}^*_{3,T-1}(0,0,x_3)&=h(x_3)\nonumber
\end{align} and $\bar{u}^*_{1,T-1}(x_1,0,0)+\bar{u}^*_{2,T-1}(0,x_2,0)+\bar{u}^*_{3,T-1}(0,0,x_3)=[-x_1 \, \, 0]'=u^*_{T-1}(x_1,x_2,x_3)$ and $J^*_{T-1}(x_1,x_2,x_3)=\bar{J}^*_{1,T-1}(x_1,0,0)+
\bar{J}^*_{2,T-1}(0,x_2,0)+\bar{J}^*_{3,T-1}(0,0,x_3)=h(x_1)+2h(x_2)+h(x_3)$.
So $\big\{ (A,B|_{\mathcal{E}_i},g,\alpha) \big\}_{ i \in \mathcal{I}}$ is decomposition of $(A,B,g,T)$. 
The systems associated with the family of problems $\big\{ (A,\rho_i \circ B,g,T) \big\}_{ i \in \mathcal{I}}$
\begin{align}
A|_{\mathcal{X}_1} =\left[
\begin{array}{ccc}
1 & 0 & 0 \\
0 & 0 & 0 \\
0 & 0 & 0
\end{array}
\right], \, &
\rho_1 \circ B = \left[
\begin{array}{cc}
1 & 1 \\
0 & 0 \\
0 & 0
\end{array}
\right] \nonumber \\
A|_{\mathcal{X}_2} =\left[
\begin{array}{ccc}
0 & 0 & 0 \\
0 & 1 & 0 \\
0 & 0 & 0
\end{array}
\right], \, &
\rho_2 \circ B = \left[
\begin{array}{cc}
0 & 0 \\
0 & 1 \\
0 & 0
\end{array}
\right] \nonumber \\
A|_{\mathcal{X}_3} =\left[
\begin{array}{ccc}
0 & 0 & 0 \\
0 & 0 & 0 \\
0 & 0 & 0
\end{array}
\right], \, &
\rho_3 \circ B = \left[
\begin{array}{cc}
0 & 0 \\
0 & 0 \\
0 & 1
\end{array}
\right] \nonumber.
\end{align} 
Notice that $\mathcal{R}(A|_{\mathcal{X}_i}) \subseteq\mathcal{R}(\rho_i \circ B), \, \forall i=1,2,3,$ and so
$J^*_{1,T-1}(x_1,0,0)=h(x_1)$, $J^*_{2,T-1}(0,x_2,0)=h(x_2)$, $J^*_{3,T-1}(0,0,x_3)=h(x_3)$. 
Hence, whenever $x_2 \neq 0$, 
$J^*_{T-1}(x_1,x_2,x_3)=h(x_1)+2h(x_2)+h(x_3) >h(x_1)+h(x_2)+h(x_3) =J^*_{1,T-1}(x_1,0,0) +  J^*_{2,T-1}(0,x_2,0)+ J^*_{3,T-1}(0,0,x_3)$ 
and $\big\{ (A,\rho_i \circ B,g,T) \big\}_{ i \in \mathcal{I}}$ cannot be a decomposition of $(A,B,g,T)$.
\end{example}

\color{black}
\section{Future Work}
\label{Sec:Conclusions}

The decompositions considered in this manuscript are natural and intuitive, 
but the conditions for their existence may be fairly restrictive in practice.
As such, future work will focus on the study of instances where such exact decompositions do not exist, 
but where ``small" perturbations of the original dynamics would allow them to exist.
Our focus will be on getting a handle on the difference between the exact and the approximate solutions in such a setting.

\section{Acknowledgments}

This research was supported by NSF CAREER award ECCS 0954601 
and AFOSR Young Investigator award FA9550-11-1-0118.

\bibliographystyle{IEEEtranS}
\bibliography{References}

\end{document}